\documentclass{article}
\usepackage{eurosym}
\usepackage{amssymb, amsmath, amsfonts,epsf}
\usepackage[all]{xy}

\newtheorem{theorem}{Theorem}

\newtheorem{corollary}[theorem]{Corollary}

\newtheorem{example}[theorem]{Example}

\newtheorem{proposition}[theorem]{Proposition}
\newtheorem{remark}[theorem]{Remark}

\newenvironment{proof}[1][Proof]{\noindent\textbf{#1.} }{\ \rule{0.5em}{0.5em}}

\begin{document}

\title{Strictly invariant submodules}
\author{Simion Breaz, Grigore C\u{a}lug\u{a}reanu and Andrey Chekhlov 
\thanks{%
invariant submodule, strongly invariant submodule, Abelian group, strictly
invariant subgroup}}
\maketitle

\begin{abstract}
If $M$ is an $R$-module, we study the submodules $K\leq M$ with the property
that $K$ is invariant with respect to all monomorphisms $K\rightarrow M$.
Such submodules are called \textsl{strictly invariant}. For the case of $%
\mathbb{Z}$-modules (i.e. Abelian groups) we prove that in many situations
these submodules are invariant with respect to all homomorphisms $%
K\rightarrow M$, submodules which were called \textsl{strongly invariant}.
\end{abstract}

\section{Introduction}

Let $K$ be a submodule of a module $M$, and let $\mathcal{X}$ be a class of
homomorphisms such that $f(K)$ makes sense for all $f\in \mathcal{X}$. We
say that $K$ is invariant with respect to the class $\mathcal{X}$ if the
inclusion $f(K)\leq K$ holds for all $f\in \mathcal{X}$. For instance, $K$
is fully invariant, injective invariant, respectively characteristic, if $K$
is invariant with respect to that class $\mathcal{X}$, where $\mathcal{X}$
is $\mathrm{End}(M)$, $\mathrm{Mon}(M)$ (i.e. the set of all monic
endomorphisms of $M$), respectively $\mathrm{Aut}(M)$. In module theory
there are important classes of modules which can be characterized by the
invariance of some submodules with respect to some classes of homomorphisms.
For instance, a module $M$ is quasi-injective (pseudo-injective) if and only
if it is fully invariant (characteristic) as a submodule of the injective
hull of $M$, cf. \cite{Fa64} (respectively \cite{ESS}). We refer to \cite%
{As-et al} for some general statements about modules which are invariant
with respect to classes of endomorphisms of injective hulls.

Injective invariant subgroups of Abelian groups were termed S-characteristic
and left invariant, respectively, in \cite{bae} or \cite{gol}. These were
used in \cite{bre} for the study of (co)hopfian modules.

The submodules $K$ which are invariant with respect to $\mathcal{X}=\mathrm{%
Hom}(K,M)$ are called \textit{strongly invariant}, and these are studied in 
\cite{cal}, with a special attention to the case of Abelian groups. We will
say that the submodule $K$ of $M$ is \textit{strictly invariant} if it is
invariant with respect to the set $\mathcal{X}=\mathrm{Mon}(K,M)$ of all
monomorphisms $K\rightarrow M$. Clearly, strongly invariant submodules are
strictly invariant and strictly invariant submodules are characteristic.

For reader's convenience we mention that the same notions are discussed in
the case of non-Abelian groups in \cite{sw}, where strongly invariant
(normal) subgroups are termed \emph{homomorph containing} and strictly
invariant subgroups are termed \emph{isomorph containing}.

In the next section we study general properties of strictly invariant
submodules. Among these it is proved that the set of all strictly invariant
submodules of a module is a complete lattice with respect to the inclusion
relation, Proposition \ref{lattice}. Moreover, if the additive group of the
module has no elements of order $2$, then every strictly invariant submodule
is invariant with respect to idempotent endomorphisms, Proposition \ref{w2}.

In the third section we study strictly invariant subgroups of Abelian
groups. We mention that in Example \ref{non-strongly} it is proved that
there exist strictly invariant submodules which are not strongly invariant.
However, we were not able to construct such an example for the case of
Abelian groups. Therefore we are focussed on finding conditions (as general
as possible) on the group and/or on the subgroup, which imply that the
strictly invariant subgroups are strongly invariant, in order to argue the
enunciation of the following conjecture: \textit{every strictly invariant
subgroup of an Abelian group is strongly invariant}. Very large classes of
Abelian groups are shown to support this conjecture.

In this context we mention that in the case of Abelian groups, all
pseudo-injective groups are quasi-injective, \cite{js67}. A similar
situation occurred in \cite{bre}: denoting by $\mathcal{Q}(G)$, the family
of all subgroups $N\leq G$ such that every homomorphism $N\longrightarrow G$
extends to an endomorphism of $G$ and by $\mathcal{P}(G)$, the family of all
subgroups $N\leq G$ such that every injective homomorphism $N\longrightarrow
G$ extends to an endomorphism of $G$, though we strongly suspect that $%
\mathcal{Q}(G)=\mathcal{P}(G)$ for Abelian groups, the proof which shows
that \emph{finitely generated subgroups from }$\mathcal{P}(G)$\emph{\ are
also in }$\mathcal{Q}(G)$ was already very hard (and the general question is
still open).

Notice that \emph{for noncommutative groups} it is easy to give examples of
strictly invariant subgroups which are not strongly invariant: the dihedral
2-groups of order at least 8 and the infinite dihedral group. The order 8
group $D_{8}$ has \emph{a unique cyclic maximal subgroup} $H$ (of order 4)
which clearly is strictly but not strongly invariant in $D_{8}$. Indeed,
there are other two order 4 subgroups which are Klein, and all the other
order 2 subgroups are (clearly) cyclic.

We finally mention that, starting from \cite{bel}, Dikranian, Giordano
Bruno, Goldsmith, Salce, Virili and Zanardo defined and studied fully inert
subgroups of Abelian groups in \cite{dar} - \cite{dik2}, \cite{gol1}, \cite%
{gol2}. Replacing fully invariant subgroups by strongly invariant subgroups,
led the first and second authors to study the strongly inert subgroups of
Abelian groups in \cite{bre1}. A natural continuation of all these (kindly
suggested by the referee) would be to study the \emph{strictly inert
subgroups} and compare these with strongly inert subgroups. We postpone this
to a forthcoming paper.

All modules we consider are over a unital ring denoted $R$. $\mathbb{F}_{2}$
denotes the field with two elements and $\mathbb{Z}(2)$ the Abelian group
with two elements. For other notations for Abelian groups we refer to \cite%
{fuc1} and \cite{fuc2}.

\section{General properties}

Using the above definitions we obtain the following chart\noindent 
\begin{equation*}
\begin{array}{ccccc}
\text{\textrm{strongly-invariant}} & \mathrm{\overset{(\ast )}{%
\Longrightarrow }} & \text{\textrm{fully-invariant}} &  &  \\ 
\mathrm{\Downarrow (4)} &  & \mathrm{\Downarrow (1)} &  &  \\ 
\text{\textrm{strictly-invariant}} & \mathrm{\overset{(1)}{\Longrightarrow }}
& \text{\textrm{injective-invariant}} & \mathrm{\overset{(3)}{%
\Longrightarrow }} & \text{\textrm{characteristic}}%
\end{array}%
\end{equation*}

The following examples (the numbering corresponds to these) show that all
reversed implications fail ((2) presents fully invariant subgroup which is
not strictly invariant; as for (*), such examples are given in \cite{cal}).

First, \emph{an injective invariant subgroup which is not strictly invariant}%
.

\begin{example}
{\rm Let $G=\left\langle a_{1}\right\rangle \oplus \left\langle
a_{2}\right\rangle \oplus \left\langle a_{3}\right\rangle $\ with $%
o(a_{i})=2^{i}$, and $H=\langle 2a_{2}\rangle \oplus \langle
a_{1}+2a_{3}\rangle $. Since $G$\ is finite, characteristic and injective
invariant subgroups coincide (because injective functions from $G$\ to $G$\
are bijective). It is proved in \cite[p. 9]{fuc2} that $H$\ is
characteristic, and it is easy to see that $H$\ is not strictly invariant
(e.g. take $2a_{2}\longmapsto a_{1}$\ and $a_{1}+2a_{3}\longmapsto a_{2}$).
Moreover, it is not fully invariant. }
\end{example}

\begin{example}
{\rm If $p$\ is a prime, the subgroup $p\mathbb{Z}$\ of $\mathbb{Z}$\ is
not strictly invariant but it is fully invariant. }
\end{example}

Other examples may be found in \cite{cal} or \cite{che1}.

\bigskip

Next, a \emph{characteristic subgroup which is not injective invariant.}

\begin{example}
{\rm By \cite[Theorem 2.14]{Ar}, for every prime $p$\ there exists a
torsion free Abelian group $G$\ of rank $2$\ with endomorphism ring
isomorphic to $R=\mathbb{Z}[\sqrt{-p}]$. Since the units of $R$\ are $\pm 1$%
, it follows that all subgroups are characteristic. Moreover, $\mathbb{Q}%
\otimes R$\ is a division ring, hence all non-zero endomorphisms of $G$\ are
injective. Let $x\in G$\ be a non-zero element, and let $f$\ and $g$\ be two
endomorphisms of $G$\ which are $\mathbb{Q}$-independent in $\mathbb{Q}%
\otimes R$. Suppose $f(x)$\ and $g(x)$\ are not $\mathbb{Z}$-independent.
Then there exist two non-zero integers $m$\ and $n$\ such that $mf(x)=ng(x)$%
\ and so $mf-ng$\ is not injective. Hence $mf=ng$, a contradiction. Thus for
every non-zero element $x$\ of $G$, the subgroup $Rx$\ has to be of rank $2$%
, hence the subgroup $\langle x\rangle $\ is not injective invariant. }
\end{example}


Next, we present an example of \emph{strictly invariant submodule which is
not strongly invariant}.

\begin{example}
\label{non-strongly} {\rm Let $R$ be a ring such that there exist
non-isomorphic simple modules $S_{1}$, $S_{2}$ and $T$ such that }

\textbf{1. } {\rm the endomorphism rings of these modules are isomorphic
to $\mathbb{Z}_{2}$; }

\textbf{2.} {\rm there are non-splitting exact sequences 
\begin{equation*}
0\rightarrow S_{1}\rightarrow K\rightarrow T\rightarrow 0,\text{ and }%
0\rightarrow S_{2}\rightarrow K\rightarrow T\rightarrow 0.
\end{equation*}%
By using the pullback diagram 
\begin{equation*}
\xymatrix{ & & 0\ar[d] & 0\ar[d] & \\ & & S_2\ar[d]\ar@{=}[r] & S_2\ar[d] &
\\ 0\ar[r] & S_1\ar[r]\ar@{=}[d] & M\ar[r]^{\psi}\ar[d]^{\varphi} &
L\ar[r]\ar[d] & 0\\ 0\ar[r] & S_1\ar[r] & K\ar[r]\ar[d] & T\ar[r]\ar[d] &
0\\ & & 0 & 0 & }
\end{equation*}%
we construct a module $M$ such that the set of its submodules is $%
\{0,S_{1},S_{2},S_{1}\oplus S_{2},M\}$ and $M/S_{1}\oplus S_{2}\cong T$
(this module is also used in \cite[Lemma 2]{js}). }

{\rm Let $\varphi :M\rightarrow K$ be a non-zero homomorphism. Then $%
\varphi (S_{2})=0$. If $\varphi (S_{1})=0$ then $\varphi $ induces a
non-zero homomorphism $T\rightarrow M$, which is impossible. We obtain that $%
\varphi (S_{1})\neq 0$, and it follows that $\varphi |_{S_{1}}$ is the
inclusion map. Therefore, if $\varphi _{1},\varphi _{2}:M\rightarrow K$ are
two non-zero homomorphisms then the restriction of these homomorphisms to $%
S_{1}\oplus S_{2}$ coincide. It follows that $(\varphi _{1}-\varphi
_{2})(S_{1}\oplus S_{2})=0$, hence $\varphi _{1}=\varphi _{2}$. This way $%
\mathrm{Hom}(M,K)=\{0,\varphi \}$ and in the same way we obtain $\mathrm{Hom}%
(M,L)=\{0,\mathrm{\psi }\}$. }

{\rm It is easy to see that if $\rho :M\rightarrow K\times L$ is the
homomorphism induced by $\varphi $ and $\psi $ then $\rho $ is a
monomorphism. Since $\mathrm{Hom}(M,K\times L)\cong \mathrm{Hom}(M,K)\times 
\mathrm{Hom}(M,L)$, it follows that $\rho $ is the only monomorphism from $M$
into $K\times L$. We conclude that $\rho (M)$ is strictly invariant. Since
there exists an epimorphism M}$\mathrm{\rightarrow }$ {\rm K, and $%
K\times 0$ is not contained in $\rho (M)$, it follows that $\rho (M)$ is not
strongly invariant. }
\end{example}

\bigskip

For reader's convenience we recall the concrete example described in \cite[%
Example 3.1]{As-et al}.

Let $R=\left( 
\begin{array}{ccc}
\mathbb{F}_{2} & \mathbb{F}_{2} & \mathbb{F}_{2} \\ 
0 & \mathbb{F}_{2} & 0 \\ 
0 & 0 & \mathbb{F}_{2}%
\end{array}%
\right) $. Then the right $R$-module $M=\left( 
\begin{array}{ccc}
\mathbb{F}_{2} & \mathbb{F}_{2} & \mathbb{F}_{2} \\ 
0 & 0 & 0 \\ 
0 & 0 & 0%
\end{array}%
\right) $ satisfies the required conditions: the simple submodules are $%
S_{1}=\left( 
\begin{array}{ccc}
0 & \mathbb{F}_{2} & 0 \\ 
0 & 0 & 0 \\ 
0 & 0 & 0%
\end{array}%
\right) $ and $S_{2}=\left( 
\begin{array}{ccc}
0 & 0 & \mathbb{F}_{2} \\ 
0 & 0 & 0 \\ 
0 & 0 & 0%
\end{array}%
\right) ,$ and it is easy to see that the simple R-module $M/(S_{1}\oplus
S_{2})$ is not isomorphic to $S_{1}$ or $S_{2}$.

\bigskip

In what follows we study the basic properties of strictly invariant
submodules. First observe that \emph{strict invariance is not a transitive
property.}

\begin{example}
\label{non-transitive} {\rm In order to show this, we observe that if $S$
and $T$ are non-isomorphic simple modules, $0\rightarrow S\rightarrow
K\rightarrow T\rightarrow 0$ is a non-splitting exact sequence and $%
M=K\oplus S$, then $S\oplus 0$ is strictly invariant in $K\oplus 0$, $%
K\oplus 0$ is strictly invariant in $M$, but $S\oplus 0$ is not strictly
invariant in $M$. }

{\rm For the case of Abelian groups, consider $G=H\oplus L=\mathbb{Z}%
(2^{\infty })\oplus \mathbb{Z}(2)$ with $K=\mathbb{Z}(2)<H$. Then $K=S(H)$,
the socle, is strongly and so strictly invariant in $H$. It is not strictly
invariant in $G$, since the composition of the isomorphism $K\cong L$ with
the injection $\iota _{L}:L\longrightarrow G$ does not map $K$ into $K$.
Finally, $H$ is a fully invariant direct summand - as divisible part of $G$
- and so strongly and strictly invariant in $G$. }
\end{example}

Next, \emph{if }$H\leq L\leq M$\emph{\ and }$H$\emph{\ is strictly invariant
in }$M$\emph{\ then }$L$\emph{\ might not be strictly invariant in }$M$.

\begin{example}
{\rm It suffices to take $K$ as in Example \ref{non-transitive}, $%
M=K\oplus K$, $H=S\oplus S$ and $K=K\oplus S$. }
\end{example}

\begin{proposition}
Let $M$ be a module and let $H\leq K$ be submodules of $M$. If $H$ is strictly
invariant in $M$ and $K/H$ is strongly invariant in $M/H$, then $K$ is
strictly invariant in $M$.
\end{proposition}

\begin{proof}
Let $f:K\longrightarrow M$ be an injective homomorphism. Since $H$ is
strictly invariant in $M$, the map $\widetilde{f}:K/H\longrightarrow M/H$, $%
\widetilde{f}(k+H)=f(k)+H$ is well-defined and a homomorphism. Since $K/H$
is strongly invariant in $G/H$, $\widetilde{f}(K/H)\subseteq K/H$ which
shows that $f(K)\subseteq K$.
\end{proof}

Notice that we cannot weaken the hypothesis \textquotedblleft $K/H$ is
strongly invariant in $M/H$\textquotedblright\ only to strictly invariant,
as the example below shows .

\bigskip

\emph{If }$H\leq K\leq M$\emph{\ and }$K$\emph{\ is strictly invariant in }$%
M $\emph{\ then }$K/H$\emph{\ might not be strictly invariant in }$M/H$\emph{%
.}

\begin{example}
{\rm For instance, if $M=H\oplus K=\mathbb{Z}_{2}\oplus \mathbb{Z}_{4}$,
the socle $H+2K$ is strictly invariant in $M$ but $(H+2K)/2K=\mathbb{Z}_{2}$
is not strictly invariant in $M/2K=\mathbb{Z}_{2}\oplus \mathbb{Z}_{2}$. }
\end{example}

Further, \emph{the intersection of a family of strictly invariant submodules
is not (in general) strictly invariant.}

\begin{example}
\label{intersection} {\rm In order to prove this, we use the same module
as in Example \ref{non-transitive}. It is easy to see that the socle $%
S\oplus S$ of $M$ and $K$ are strictly invariant submodules of $M$, but $%
S\oplus 0=(S\oplus S)\cap K$ is not strictly invariant. }

{\rm For the case of Abelian groups we can consider $G=D\oplus R$, where $%
D$ is a divisible $p$-group and $R$ is a reduced $p$-group. Then $D$ and $%
G[p]=D[p]\oplus R[p]$ are strongly invariant subgroups; however, the
subgroup $D\cap G[p]=D[p]$ is not strictly invariant in $G$ (this covers the
missing example in \cite{cal}, where an example of two strongly invariant
subgroups with not strongly invariant intersection was not given). }

{\rm Intersections of strictly invariant subgroups may not be strictly
invariant also in torsion-free groups. To see this we use Example 2 (p. 107, 
\cite{che1}). We recall some details about this example. }

{\rm Let $E_{1}$, $E_{2}$, $E_{3}$ and $E_{4}$ be torsion-free groups of
rank 1, let $p$, $q$, $p_{2}$ and $p_{3}$ be distinct primes, let the types
of the groups $E_{1}$, $E_{2}$, and $E_{3}$ be pairwise incomparable, and
let $E_{1}\cong E_{4}$, $p_{2}E_{2}=E_{2}$, $p_{3}E_{3}=E_{3}$, $pE_{1}\neq
E_{1}$, $pE_{2}\neq E_{2}$, $pE_{3}\neq E_{3}$, $p_{2}E_{1}\neq E_{1}$, $%
p_{2}E_{3}\neq E_{3}$, $p_{3}E_{1}\neq E_{1}$, $p_{3}E_{2}\neq E_{2}$, $%
qE_{1}\neq E_{1}$, $qE_{2}\neq E_{2}$, $qE_{3}\neq E_{3}$. }

{\rm A group $G$ is constructed as subgroup of a divisible torsion-free
group, using a vector space over the field of rational numbers. Write $%
A=\left\langle E_{1},E_{2},p^{-\infty }(e_{1}+e_{2})\right\rangle $, $%
B=\left\langle E_{3},E_{4},q^{-\infty }(e_{3}+e_{4})\right\rangle $ and $%
G=A\oplus B$ where $0\neq e_{i}\in E_{i}$, $i\in \{1,2,3,4\}$ and $%
p^{-\infty }a$ is the infinite set $p^{-1}a$, $p^{-2}a$,... If $\mathbf{t}%
(E_{i})$ denotes the type of $E_{i}$, it is shown that $A$ and $E_{1}\oplus
E_{4}=G(\mathbf{t}(E_{1}))$ are strongly invariant in $G$ but $E_{1}=A\cap G(%
\mathbf{t}(E_{1}))\cong E_{4}$ is not strictly invariant. }
\end{example}

\bigskip

In the sequel we prove some basic properties of strictly invariant
submodules. We denote by $\mathcal{T}(M)$ the set of all strictly invariant
submodules of $M$.

\begin{proposition}
\label{lattice} Let $M$ be an $R$-module. If $\{S_{i}\}_{i\in I}$ is a
family of submodules from $\mathcal{T}(M)$ then $\sum_{i\in I}S_{i}\in 
\mathcal{T}(M)$. Consequently, $(\mathcal{T}(M),\subseteq )$ is complete
lattice.
\end{proposition}

\begin{proof}
Let $\{S_{i}\}_{i\in I}$ be a family of strictly invariant submodules of a
module $M$ and let $f:\sum\limits_{i\in I}S_{i}\longrightarrow M$ be an
injective homomorphism. Denoting by $\iota _{i}:S_{i}\longrightarrow
\sum\limits_{i\in I}S_{i}$ ($i\in I$) the inclusions, the compositions $%
f\circ \iota _{i}:S_{i}\longrightarrow G$ are also injective. By hypothesis, 
$(f\circ \iota _{i})(S_{i})\subseteq S_{i}$ and so $f(\sum\limits_{i\in
I}S_{i})\subseteq \sum\limits_{i\in I}S_{i}$, as required.

The existence of inf's now follows because an ordered set $A$ is a complete
lattice if and only if for every subset $B\subseteq A$, there exists $\sup B$%
.
\end{proof}

\medskip

Using Example \ref{intersection}, we observe that in general the infimum of
a family $\{S_{i}\}_{i\in I}$ from $\mathcal{T}(M)$ is not the intersection
of these submodules, that is, the complete lattice $(\mathcal{T}%
(M),\subseteq )$ above is not a complete sublattice of the lattice of all
submodules of $M$.

\bigskip

Let $M$ be an $R$-module. If $K\leq M$, we denote by $\mathcal{M}_{M}(K)$
the sum of all submodules $f(K)$, where $f$ ranges all monomorphisms $%
f:K\rightarrow M$. We denote by $\mathcal{S}(M)$ the lattice of all
submodules of $M$ and $\mathcal{M}_{M}(\mathcal{S}(M))=\{\mathcal{M}%
_{M}(K):K\leq M\}$.

\begin{proposition}
Let $M$ be an $R$-module. Then $\mathcal{M}_{M}(-):\mathcal{S}(M)\rightarrow 
\mathcal{S}(M)$ is an idempotent decreasing operator, and $\mathcal{M}_{M}(%
\mathcal{S}(M))=\mathcal{T}(M)$.
\end{proposition}

\begin{proof}
If $f\colon \mathcal{M}_{M}(K)\rightarrow M$ is an injective homomorphism,
then for every monomorphism $g:K\rightarrow M$, since $g(K)\leq \mathcal{M}%
_{M}(K)\leq M$, we can consider $f\circ g:K\rightarrow M$, which is also a
monomorphism. Then $(f\circ g)(K)\leq \mathcal{M}_{M}(K)$, and we conclude
that $f(\mathcal{M}_{M}(K))\leq \mathcal{M}_{M}(K)$. Finally for every
submodule $K$ of $M$ we have $\mathcal{M}_{M}(K)\in \mathcal{T}(M)$ and the
surjectivity follows from the fact that $K\in \mathcal{T}(M)$ clearly
implies $\mathcal{M}_{M}(K)=K$.
\end{proof}

\begin{corollary}
Let $M$ be an $R$-module and $K\leq M$.

\begin{enumerate}
\item If $f:K\rightarrow M$ is a homomorphism and $f(K)\nsubseteq \mathcal{M}%
_{M}(K)$ then for every $\alpha :K\rightarrow \mathcal{M}_{M}(K)$ there
exists $x\in K$ such that $f(x)=\alpha (x)$.

\item $f(K)\cap \mathcal{M}_{M}(K)\neq 0$ for every $0\neq f\in \mathrm{Hom}%
\,(K,G)$.

\item If $H\leq M$ is a submodule of $M$ such that $H\cap \mathcal{M}%
_{M}(K)=0$, then $\mathrm{Hom}\,(H,K)=0$.
\end{enumerate}
\end{corollary}

\begin{proof}
1. The image of $f-\alpha $ is not contained in $\mathcal{M}_{M}(K)$. Then $%
f-\alpha $ is not a monomorphism.

For 2 and 3 we apply 1 taking for $\alpha $ the inclusion map.
\end{proof}

\begin{corollary}
\label{ch} Let $H$ be a strictly invariant submodule of $M$. Then:

\begin{enumerate}
\item $f(H)\subseteq H$ for every non-zero homomorphism $f:H\rightarrow G$
such that $f(f(H)\cap H)=0$.

\item $f(H)\cap H\neq 0$ for every $0\neq f\in \mathrm{Hom}\,(H,M)$.

\item $\mathrm{Hom}(H,L)=0$ for every $L\leq M$ such that $L\cap H=0$.
\end{enumerate}
\end{corollary}

\begin{proof}
1. Suppose there exists $f:H\rightarrow G$, $f\neq 0$, such that $f(f(H)\cap
H)=0$ and consider $\overline{f}:H\rightarrow G$, $\overline{f}(h)=h+f(h)$
for every $h\in H$. If $h+f(h)=0$ then $h\in f(H)\cap H$, hence $f(h)=0$ and
so $h=0$. Therefore, $\overline{f}$ is a monomorphism and $f(H)\subseteq H$
by strictly invariance.

The statements 2 and 3 are consequences of 1.
\end{proof}

\begin{corollary}
\label{direct-summands} A direct summand is strictly invariant if and only
if it is fully invariant.
\end{corollary}

For any pair $A$, $N$ of modules, denote by $S_{A}(N)=\sum_{f\in \mathrm{Hom}%
(N,A)}f(N)$ the $N$-\textsl{socle} of $A$, a submodule of $A$.

\begin{proposition}
\label{sd} Let $M=A\oplus B$ be an $R$-module. If $K\leq A$ and $L\leq B$
are submodules such that $K\oplus L$ is strictly invariant in $M$ then

\begin{enumerate}
\item $K$ is strictly invariant in $A$.

\item $L$ is strictly invariant in $B$.

\item $S_{A}(L)\leq K$ and $S_{B}(K)\leq L$.
\end{enumerate}
\end{proposition}

\begin{proof}
Let $f:K\rightarrow A$ be a monomorphism. Then $\overline{f}=f\oplus \iota
_{L}:K\oplus L\rightarrow A\oplus B$ is a monomorphism ($\iota
_{N}:L\rightarrow B$ denotes the inclusion map) and, since $K\oplus L$ is
strictly invariant, the inclusion $f(K)\leq K$ follows.

Any homomorphism $f:L\rightarrow A$ induces a homomorphism $\overline{f}%
:K\oplus L\rightarrow A\oplus B$, $\overline{f}(k+\ell )=k+f(\ell )+\ell $
for every $k\in K$, $\ell \in L$. If $k+f(\ell )+\ell =0$ then $\ell =0$ and
we also obtain $k=0$. Therefore, $\overline{f}$ is injective, and it is easy
to see that $f(\ell )\in K$ for every $\ell \in L$.
\end{proof}

In the following result, for any (finite or infinite) cardinal $k$, $M^{(k)}$
denotes the direct sum of $k$ copies of $M$.

\begin{corollary}
\label{+} Let $H$ be a strictly invariant submodule of $M$. Then the
following conditions are equivalent:

$1)$ $H^{2}$ is a strictly invariant submodule of $M^{2}$.

$2)$ $H^{(k)}$ is a strictly invariant submodule of $M^{(k)}$ for every
cardinal number $k$.

$3)$ $H^{(k)}$ is a strongly invariant submodule of $M^{(k)}$ for every
cardinal number $k$.

$4)$ $H$ is a strongly invariant submodule of $M$.
\end{corollary}

\begin{proof}
1) $\Rightarrow $ 4) Follows from Proposition \ref{sd}.

4) $\Rightarrow $ 3) Since $H$ is strongly invariant in $M$, by~\cite{cal}, $%
H^{(k)}$ is strongly invariant in $M^{(k)}$ for any finite $k$. For any
infinite cardinal $k$, since every element of $M^{(k)}$ belongs to a direct
summand isomorphic to $M^{(n)}$ for some finite $n$, $H^{(k)}$ is also
strongly invariant in $M^{(k)}$.

3) $\Rightarrow $ 2) Obvious.

2) $\Rightarrow $ 1) Obvious.
\end{proof}

\begin{proposition}
\label{w2} Let $M=A\oplus B$ be a module such that the additive group $A$
has no elements of order $2$. If a submodule $H\leq M$ is strictly invariant
then there exist $K\leq A$, $L\leq B$ such that

$1)$ $H=K\oplus L$.

$2)$ $K$ is strictly invariant in $A$ and $L$ is strictly invariant in $B$.

$3)$ $S_{A}(L)\leq K$ and $S_{B}(K)\leq L$.
\end{proposition}

\begin{proof}
By Proposition \ref{sd}, it suffices to prove that $H=K\oplus L$ with $K\leq
A$ and $L\leq B$.

Let $\pi _{A}:M\rightarrow A$ and $\pi _{B}:M\rightarrow B$ be the
projections and suppose $H$ is strictly invariant in $M$. 
If $\pi \in \{\pi _{A},\pi _{B}\}$ and $\pi (H)\nleq H$, there is $0\neq
h\in H$ such that $\pi (h)\notin H$. Therefore $(\pi +1)(h)\notin H$ and so
the restriction $(\pi +1)|_{H}$ is not injective (otherwise $H$ is not
strictly invariant). Hence $\ker \,((\pi +1)|_{H})\neq 0$ and if $\pi (x)=-x$
for $0\neq x\in H$ then $\pi (x)=-\pi (x)$, i.e. $2\pi (x)=0$, a
contradiction.

It follows that $\pi _{A}(H)\leq H$, and similarly $\pi _{B}(H)\leq H$.
Therefore $H=M\oplus N$, where $M=\pi _{A}(H)\leq A$ and $N=\pi _{B}(H)\leq
B $.
\end{proof}

\begin{remark}
\textrm{The previous proposition is not valid if both $A$ and $B$ have
elements of order $2$. }
\end{remark}

This follows from the construction used in Example \ref{non-strongly}.

\section{Strictly invariant subgroups}

As mentioned in the Introduction, for fairly large classes of groups, we
show that our conjecture, \textquotedblleft strictly invariant subgroups of
Abelian groups are strongly invariant\textquotedblright , holds.

We start the investigation of strictly invariant subgroups of Abelian groups
with a consequence of Proposition \ref{sd} (in this section, unless
otherwise stated, \textquotedblleft group\textquotedblright\ means
\textquotedblleft Abelian group\textquotedblright ). By $\mathbb{P}$ we
denote the set of all prime numbers and for an Abelian group $G$ and a prime 
$p$, $G_{p}=\{x\in G:\exists n\in \mathbb{N},p^{n}x=0\}$ denotes the $p$%
-component of $G$. For an element $x\in G$, the $p$\textsl{-height of} $x$,
denoted $h_{p}(x)$, is the smallest integer $n$ such that $x\in p^{n}H$. If $%
x\in p^{n}H$ for all positive integers $n$ then we say that $x$ is of
infinite height.

If $G$ is an Abelian group, we denote by $D(G)$ its divisible part.
Moreover, if $p$ is a prime then $D_{p}(G)$ denotes the $p$-component of $%
D(G)$.

\begin{corollary}
\label{divisible-part} Let $H$ be a strictly invariant subgroup of a group $%
G $. Then

\begin{enumerate}
\item $D(H)=D(G)$ whenever $D(H)$ is not a torsion group.

\item $D_{p}(H)=D_{p}(G)$ whenever $p$ is a prime and $D_{p}(H)\neq 0$.
\end{enumerate}
\end{corollary}

\begin{proof}
We chose a decomposition $H=H_{0}\oplus D(H)$. Using \cite[Theorem 21.2]%
{fuc1}, we can find a direct decomposition $G=K\oplus D(H)$ such that $%
H_{0}\leq K$. By Proposition \ref{sd}, $S_{K}(D(H))\leq H_{0}$. Since $H_{0}$
is reduced and every image of a divisible group is divisible, it follows
that $S_{K}(D(H))=0$.

1. If $D(H)$ is not torsion then it has a direct summand isomorphic to $%
\mathbb{Q}$, hence for every non-reduced group $L$ we have non-zero
homomorphisms $D(H)\rightarrow L$. It follows that $K$ is reduced, hence $%
D(H)=D(G)$.

2. If $D_{p}(H)\neq 0$ then $D(H)$ has a direct summand isomorphic to $%
\mathbb{Z}(p^{\infty })$. From $S_{K}(D(H))=0$ it follows that $D_{p}(K)=0$,
hence $D_{p}(H)=D_{p}(G)$.
\end{proof}

\begin{theorem}
\label{str-inv-p} Let $G$ be a group and let $H$ be a $p$-subgroup of $G$.
Then $H$ is strictly invariant in $G$ if and only if it satisfies one of the
following conditions:

\begin{enumerate}
\item $H=G_{p}$.

\item there exists a non-negative integer $n$ such that $H=G[p^{n}]$.

\item there exists a non-negative integer $n$ such that $H=G[p^{n}]+D_{p}(G)$%
.
\end{enumerate}
\end{theorem}

\begin{proof}
Suppose $H$ is strictly invariant. Since $H$ is a $p$-group, we can suppose
that $G$ is also a $p$-group.

As in the proof of Corollary \ref{divisible-part} we can find direct
decompositions $G=K\oplus D(H)$ and $H=H_{0}\oplus D(H)$ with $H_{0}\leq K$.
Using Proposition \ref{sd} it follows that $H_{0}$ is strictly invariant in $%
K$. Therefore, we can assume w.l.o.g. that $H$ is reduced.

\textit{Case I: $H$ is not bounded.} We will prove that $H=G$. Let us fix an
element $y\in G$. We can assume w.l.o.g that there exists $u$, the smallest
positive integer such that $p^{u}y\in H$.

If $p^uy=0$, we chose $\langle x\rangle$ a direct summand of $H$ such that $%
\mathrm{ord}(x)\geq\mathrm{ord}(y)$. If $H=\langle x\rangle \oplus L$ then $%
\langle x+y\rangle +L=\langle x+y\rangle \oplus L$ and $\mathrm{ord}(x+y)=%
\mathrm{ord}(x)$. It follows that $\langle x+y\rangle +L\cong H$, and we
obtain that $x+y\in H$, hence $y\in H$.

Suppose that $p^{u}y\neq 0$ and we can find non-zero elements in $\langle
p^{u}y\rangle $ whose heights computed in $H$ are infinite. Let $\oplus
_{n>0}B_{n}$ be a basic $p$-subgroup of $H$, where for every $n$ the group $%
B_{n}$ is isomorphic to a direct sum of cyclic groups of order $p^{n}$ (see 
\cite[Theorem 32.4]{fuc1}). Since $H$ has an unbounded basic subgroup, there
exists $n>0$ such that $p^{n}>\mathrm{ord}(y)$ and $B_{n}\neq 0$. Since all
non-zero elements of $B_{n}$ are of finite height, it follows that $\langle
p^{u}y\rangle \cap B_{n}=0$, hence we can find a $B_{n}$-high subgroup $%
C\leq H$ such that $\langle p^{u}y\rangle \leq C$. By the proof of \cite[%
Proposition 27.1]{fuc1} it follows that $H=B\oplus C$. Therefore, there
exists a decomposition $H=\langle x\rangle \oplus L$ such that $\mathrm{ord}%
(x)\geq \mathrm{ord}(y)$ and $p^{u}y\in L$. Write $H=\langle x\rangle \oplus
L$, and consider the homomorphism $f:H\rightarrow G$, $f(mx+\ell
)=mx+my+\ell $, for all $m\in \mathbb{Z}$ and $\ell \in L$. If $kx+ky+\ell
=0 $, it follows that $ky\in H$, hence $p^{u}$ divides $k$. Therefore $%
ky+\ell \in L$, hence $kx=0$. Since $\mathrm{ord}(x)\geq \mathrm{ord}(y)$ we
obtain $ky=0$, and it follows that $\ell =0$. Finally, $f$ is a monomorphism
and it is easy to conclude that $y\in H$.

If $p^{u}y\neq 0$ and all non-zero elements of $\langle p^{u}y\rangle $ are
of finite heights (computed in $H$) then by \cite[Theorem 33.4]{fuc1}, there
exists a basic subgroup $B$ of $H$ such that $\langle p^{u}\rangle \leq B$.
Since $B$ is unbounded, we can find a cyclic direct summand $\langle
x\rangle $ of $B$, hence of $G$, such that $\mathrm{ord}(x)\geq \mathrm{ord}%
(y)$, and $\langle x\rangle \cap \langle p^{u}y\rangle =0$. Then there
exists a decomposition $H=\langle x\rangle \oplus L$ such that $p^{u}y\in L$%
, and we can repeat the proof used in the previous case to conclude that $%
y\in H$.

\textit{Case II: $H$ is bounded.} If $p^{n}=\exp H$ then clearly $H\leq
G[p^{n}]$. Assume that $H<G[p^{n}]$ and let $x\in H$ be such that $\mathrm{%
ord}(x)=p^{n}$. Since $G[p^{n}]$ is generated by the elements of order $%
p^{n} $ in $G$, there exists $y\in G[p^{n}]$ such that $\mathrm{ord}%
(y)=p^{n} $ and $y\notin H$. By \cite[Theorem 27.1]{fuc1}, there exist $%
K,L\leq G$ such that $G[p^{n}]=\left\langle x\right\rangle \oplus
K=\left\langle y\right\rangle \oplus L$. By Dedekind's law, $H=H\cap
G[p^{n}]=\left\langle x\right\rangle \oplus (H\cap K)$. Since $L\cong K$,
there exists $L_{1}\leq L $ such that $L_{1}\cong H\cap K$ and it follows
that $\left\langle y\right\rangle \oplus L_{1}$ and $\left\langle
x\right\rangle \oplus (H\cap K)=H$ are isomorphic. Since $y\notin H$ this
contradicts the fact that $H$ is strictly invariant in $G$. Thus $H=G[p^{n}]$%
, as desired.

As for the converse, it is enough to observe that if $H$ verifies any of the
conditions 1--3 then it is strongly invariant.
\end{proof}

\begin{remark}
\textrm{For the case $p\neq 2$ the above result can also be proved by using
Proposition \ref{w2}. }
\end{remark}

Namely, if $H$ has an unbounded basic subgroup, we write $G=D_{p}(G)\oplus R$
with reduced $R$ and $H=D_{p}(G)\oplus K$ with $K\leq R$. Let $k$ be a
positive integer. There exists a cyclic direct summand $C\cong \mathbb{Z}%
(p^{n})$ of $R$ with $n\geq k$. Applying Proposition \ref{w2} for the direct
decomposition $G=(D_{p}(G)\oplus C)\oplus L$, it follows that $%
S_{L}(D_{p}(G)\oplus C)=S_{L}(C)\leq H$. It is easy to see that $%
G[p^{n}]\leq H$, hence $G[p^{k}]\leq H$ for all $k$ and the proof is
complete.\textrm{\ }

\bigskip

From now on, starting with the next corollary, the results are all in the
line of the conjecture, stating that every strictly invariant subgroup of an
Abelian group is strongly invariant.

First we are able to show that

\begin{corollary}
\label{che} Every torsion strictly invariant subgroup of any group is
strongly invariant.
\end{corollary}

\begin{proof}
It is proved in \cite{cal} that a torsion subgroup is strongly invariant if
and only if all its primary components are strongly invariant. Using Theorem %
\ref{str-inv-p}, the conclusion now follows. Indeed, all subgroups from the
theorem, $G_{p}$, $G[p^{n}]$ and $G[p^{n}]+D_{p}(G)$ are strongly invariant.
\end{proof}

\bigskip

In the following proposition, $r_{p}(K)$ denotes the $p$-rank of $K$. We
will prove that if the divisible part of a group is large enough, then all
strictly invariant subgroups are strongly invariant.

\begin{proposition}
Let $G=D(G)\oplus R$ be a group and $r_{p}(D(G))\geq \max \{r_{p}(R),\aleph
_{0}\}$ for every $p\in \mathbb{P}\cup \{0\}$. Then every strictly invariant
subgroup of $G$ is strongly invariant.
\end{proposition}

\begin{proof}
Let $H$ be a strictly invariant subgroup of $G$. Then $H=D(G)\oplus K$ (we
use Corollary \ref{divisible-part}) and we can suppose that $K\leq R$. It
suffices to prove that $K$ is strongly invariant in $R$. In order to do
this, let us fix a homomorphism $f:K\rightarrow R$.

By the rank hypotheses it follows that $K$ can be embedded in $D(G)$ and $%
D(G)\cong D(G)\oplus D(G)$. Therefore, there exists a monomorphism $\alpha
:H\rightarrow D(G)$. We consider the homomorphism $g:H\rightarrow G$,
defined by $g(d+k)=\alpha (d+k)+f(k)$ for all $d\in D(G)$ and $k\in K$. It
is easy to see that $g$ is a monomorphism. Therefore $f(k)\in K$ for all $%
k\in K$ and the proof is complete.
\end{proof}

\bigskip

The following results refer to torsion-free subgroups or torsion-free groups.

\begin{proposition}
Let $H$ be a strictly invariant torsion-free subgroup of a group $G$. If $H$
is of finite rank then it is strongly invariant.
\end{proposition}

\begin{proof}
Let $f:H\rightarrow G$ be a homomorphism. We claim that there exists a
positive integer $k$ such that for all $x\in H$ we have $f(x)\neq kx$.

By contradiction suppose that the above claim fails. Then for every positive
integer $k$ there exists $x_{k}\in H$ such that $f(x_{k})=kf(x_{k})$. We
will prove by induction on the cardinality of $S$ that every non-empty
finite subset $S\subseteq \{x_{k}:k\in \mathbb{N}^{\star }\}$ is linearly
independent. Since for $|S|=1$ the property is obvious, suppose that all
non-empty subsets $S\subseteq \{x_{k}:k\in \mathbb{N}^{\star }\}$ of
cardinality at most $n$ are linearly independent. Let $\{x_{k_{1}},\dots
,x_{k_{n+1}}\}$ be a subset of cardinality $n+1$, and suppose that there
exist integers $\alpha _{1},\dots ,\alpha _{n+1}$ such that $%
\sum_{i=1}^{n+1}\alpha _{i}x_{k_{i}}=0$. Applying $f$ we obtain $%
\sum_{i=1}^{n+1}k_{i}\alpha _{i}x_{k_{i}}=0$, and so $%
\sum_{i=1}^{n}(k_{i}-k_{n+1})\alpha _{i}x_{k_{i}}=0$. By the induction
hypothesis it follows that $(k_{i}-k_{n+1})\alpha _{i}=0$ for all $i\in
\{1,\dots ,n\}$, and now it is easy to conclude that $\{x_{k_{1}},\dots
,x_{k_{n+1}}\}$ is linearly independent. Hence the rank of $H$ is infinite,
a contradiction.

Let $k$ be a positive integer such that for all $x\in H$ we have $f(x)\neq
kx $. Then the map $g:H\rightarrow G$, $g(x)=kx+f(x)$ is a monomorphism.
Using the strictly invariance of $H$, it follows that $g(H)\subseteq H$,
hence $f(H)\subseteq H$, and the proof is complete.
\end{proof}

\begin{proposition}
If $G$ is torsion-free and all rank 2 pure subgroups of $G$ are
indecomposable, then every strictly invariant subgroup of $G$ is strongly
invariant.
\end{proposition}

\begin{proof}
Suppose there exists a non-injective homomorphism $f:H\rightarrow G$. Then
there is a non-zero element $x\in H$ such that $f(x)=x$. Indeed, let $f\in 
\mathrm{Hom}(H,G)$ and $f(H)\nleq H$. If $g$ is the embedding of $H$ in $G$
then $(f-g)H\nleq H$, so there exists a non-zero $x\in H$ such that $f(x)=x$%
. Take $y$ a non-zero element from the kernel of $f$, and let $L$ be the
pure subgroup generated by $x$ and $y$. For every non-zero element $z\in L$,
we have a relation $kz=mx+ny$ with $k\neq 0$. Then $kf(z)=mx$, and so $%
kf^{2}(z)=mx=kf(z)$. From the torsion-free hypothesis, we can view $f$ as an
idempotent endomorphism of $L$ whose image has rank $1$. It follows that $L$
is not indecomposable, a contradiction.~
\end{proof}

\bigskip

Examples of such groups include the purely indecomposable groups determined
by Griffith in the reduced case (see Theorem \textbf{88.5} \cite{fuc2}) and
in particular the so-called cohesive groups considered by Dubois (see
Exercise 17, \S\ \textbf{88}, \cite{fuc2}).

\bigskip

Next, it is easy to see that the only strictly invariant subgroups of rank 1
torsion-free groups are the trivial ones (i.e. $0$ or the whole group). This
is clear for $\mathbb{Z}$ and follows from Corollary \ref{divisible-part}
for $\mathbb{Q}$. By Theorem \ref{str-inv-p}, this also holds for $\mathbb{Z}%
(p)^{\mathbb{N}}$.

\begin{proposition}
\label{comp-dec} A subgroup of a completely decomposable group is strictly
invariant if and only if it is a fully invariant direct summand.
\end{proposition}

\begin{proof}
Let $G=\oplus _{i\in I}G_{i}$ be a completely decomposable group, where all
groups $G_{i}$ are of rank $1$. If $H$ is a strictly invariant subgroup of $%
G $, then using Proposition \ref{w2} it follows that for all $i\in I$ we
have $\pi _{i}(H)\leq H$, where $\pi _{i}:G\rightarrow G_{i}$ denotes the
projection. Then $H=\oplus _{i\in I}(H\cap G_{i})$, and, using Proposition %
\ref{sd}, $H\cap G_{i}$ is a strictly invariant subgroup of $G_{i}$ for
every $i\in I$. By the preceding paragraph, $H=\oplus _{j\in J}G_{j}$, where 
$J$ is the set of all $j\in I$ such that $H\cap G_{j}=G_{j}$. The conclusion
is now a consequence of Corollary \ref{direct-summands}.
\end{proof}

\begin{corollary}
Let $G$ be a separable torsion-free group and $H$ a nonzero strictly
invariant subgroup. Then $H$ is strongly invariant.
\end{corollary}

\begin{proof}
Let $f:H\longrightarrow G$ be a homomorphism. If $x\in H$ then there exists
a finite rank completely decomposable $G_{1}\oplus \dots \oplus G_{n}$
direct summand of $G$ such that $x,f(x)\in G_{1}\oplus \dots \oplus G_{n}$.
Using Proposition \ref{w2} it follows that $K=H\cap (G_{1}\oplus \dots
\oplus G_{n})$ is a strictly invariant subgroup of $G_{1}\oplus \dots \oplus
G_{n}$. Then by Proposition \ref{comp-dec}, $K$ is strongly invariant,
whence $f(x)\in H$.
\end{proof}

\bigskip

Finally, we show that the \emph{groups, all whose subgroups are strictly
invariant}, coincide with those all whose subgroups are strongly invariant.

\begin{theorem}
\label{all}The only groups in which every subgroup is strictly invariant are
the direct sums of cocyclic groups, at most one, for each prime number.
\end{theorem}

\begin{proof}
The proof in \cite{cal} holds verbatim with obvious changes for torsion
groups. Indeed, a subgroup $H$ of a torsion group $G$ is strictly invariant
if and only if the $p$-component $H_{p}$ is strictly invariant in $G_{p}$
for each prime $p$, in a $p$-group $G$ every subgroup is strictly invariant
if and only if $G$ is cocyclic, and, in a torsion group every subgroup is
strictly invariant if and only if each $p$-component has this property.

Further, there are no torsion-free nor (genuine) mixed groups with only
strictly invariant subgroups. Indeed, using the multiplication with $\frac{1%
}{p}$ for a suitable prime $p$, it is easy to see that rank 1 torsion-free
groups are strictly invariant simple (i.e. have only trivial strictly
invariant subgroups). Therefore, no torsion-free groups have only strictly
invariant subgroups.

As for (genuine) mixed groups $G$, the torsion part $T(G)$ must be as in the
strongly invariant case and so is a direct summand (see \cite{cal}). If $%
G=T(G)\oplus F$, we continue with $F$ as above.
\end{proof}

\bigskip

\noindent\textbf{Acknowledgement.}
Thanks are due to the referee, for corrections and suggestions which
improved our presentation.

\bigskip

S. Breaz, Department of Mathematics, Babe\c{s}-Bolyai University, 1 Kog\u{a}%
lni-ceanu Street, Cluj-Napoca, Romania

E-mail: bodo@math.ubbcluj.ro

G. C\u{a}lug\u{a}reanu, Department of Mathematics, Babe\c{s}-Bolyai
University, 1 Kog\u{a}lniceanu Street, Cluj-Napoca, Romania

E-mail: calu@math.ubbcluj.ro

A. Chekhlov, Faculty of Mechanics and Mathematics, Tomsk State University,
Tomsk, Russia

E-mail: cheklov@math.tsu.ru


\bigskip

\begin{thebibliography}{99}
\bibitem{Ar} D. M. Arnold \textsl{Finite rank torsion free Abelian groups
and rings}. Lecture Notes in Mathematics 931. Berlin-Heidelberg-New York.
Springer-Verlag, 1982.

\bibitem{bae} R. Baer \textsl{Groups without isomorphic proper subgroups}.
Bull. Amer. Math. Soc. \textbf{50} (4) (1944), 267-278.

\bibitem{bel} V. V. Belyaev, M. Kuzucuoglu, E. Seckin \textsl{Totally inert
groups}. Rend. Sem. Mat. Univ. Padova \textbf{102} (1999), 151-156.

\bibitem{bre} S. Breaz, G. C\u{a}lug\u{a}reanu, P. Schultz \textsl{Subgroups
which admit extensions of homomorphisms}. Forum Mathematicum \textbf{27} (5)
(2015), 2533-2549.

\bibitem{bre1} S. Breaz, G. C\u{a}lug\u{a}reanu \textsl{Strongly inert
subgroups of Abelian groups}. Rend. Sem. Mat. Univ. Padova \textbf{138}
(2018), 101-114.

\bibitem{cal} G. C\u{a}lug\u{a}reanu \textsl{Strongly invariant subgroups of
Abelian groups}. Glasgow J. Math. \textbf{57} (2) (2015), 431-443.


\bibitem{che1} A. R. Chekhlov \textsl{On Strongly Invariant Subgroups of
Abelian Groups}. Mathematical Notes \textbf{102} (1) (2017), 105-110.

\bibitem{dar} U. Dardano, D. Dikranjan, S. Rinauro. \textsl{Inertial
properties in groups}. International J. of Group Theory \textbf{7} (3)
(2018), 17-62.

\bibitem{dik} D. Dikranjan, A. Giordano Bruno, L. Salce, S. Virili. \textsl{%
Intrinsic algebraic entropy}. J. of Pure and Applied Algebra \textbf{219}
(7) (2015), 2933-2961.

\bibitem{dik1} D. Dikranjan, A. Giordano Bruno, L. Salce, S. Virili. \textsl{%
Fully inert subgroups of divisible Abelian groups}. J. Group Theory \textbf{%
16} (6) (2013), 915-940.

\bibitem{dik2} D. Dikranjan, L. Salce, P. Zanardo. \textsl{Fully inert
subgroups of free Abelian groups}. Periodica Math. Hungarica \textbf{69}
(2014), 69-78.

\bibitem{ESS} N. Er, S. Singh, A. K. Srivastava \textsl{Rings and modules
which are stable under automorphisms of their injective hulls}. J. Algebra 
\textbf{379} (2013), 223--229.

\bibitem{Fa64} C. Faith \textsl{Lectures on Injective Modules and Quotient
Rings}. Lecture Notes in Mathematics 49, Springer Verlag, 1967.

\bibitem{fuc1} L. Fuchs \textsl{Infinite Abelian Groups}. vol. 1 Academic
Press (1970).

\bibitem{fuc2} L. Fuchs \textsl{Infinite Abelian Groups}. vol. 2 Academic
Press (1973).

\bibitem{gol} B. Goldsmith, K. Gong \textsl{A note on Hopfian and co-Hopfian
Abelian groups}. (2012), Proceedings of the Conference on Group Theory and
Model Theory, Contemporary Maths. Series, editors: M. Droste, L. Fuchs, L.
Str\"{u}ngmann, K. Tent.

\bibitem{gol1} B. Goldsmith, L. Salce, P. Zanardo. \textsl{Fully inert
submodules of torsion-free modules over the ring of p-adic integers}.
Colloquium Math. \textbf{136} (2014), 169-178.

\bibitem{gol2} B. Goldsmith, L. Salce, P. Zanardo. \textsl{Fully inert
subgroups of Abelian p-groups}. J. of Algebra \textbf{419} (2014), 332-349.

\bibitem{As-et al} P. A. Guil Asensio, T. C. Quynh, A. K. Srivastava \textsl{%
Additive unit structure of endomorphism rings and invariance of modules}.
Bull. Math. Sci. \textbf{7} (2017), 229-246.

\bibitem{js} S. K. Jain, S. Singh \textsl{Quasi-injective and
pseudo-injective modules}. Can. Math. Bull. \textbf{18} (1975), 359-366.

\bibitem{js67} S. K. Jain, S. Singh \textsl{On pseudo-injective modules and
self-pseudo-injective rings}. J. Math. Sciences \textbf{2} (1967), 23-31.




\bibitem{sw} http://groupprops.subwiki.org/wiki/Homomorph-containing%
\_subgroup
\end{thebibliography}
\end{document}